\newif\ifAlesStyle

\AlesStyletrue
\AlesStylefalse

\ifAlesStyle

\documentclass[reqno]{amsart}

\else

\documentclass[reqno,12pt]{amsart}
\usepackage[left=2cm,right=2cm,top=1.5cm,bottom=2cm]{geometry}

\fi

\usepackage{amssymb}

\newcommand{\F}{\mathbb{F}}
\def\chr{\operatorname{char}}
\def\m{^{-1}}

\newcommand{\id}{\operatorname{id}}
\newcommand{\orth}{\psi}
\newcommand{\nsq}{\zeta}
\newcommand{\eps}{\varepsilon}

\numberwithin{equation}{section}
\newtheorem{thm}{Theorem}[section]
\newtheorem{lem}[thm]{Lemma}
\newtheorem{cor}[thm]{Corollary}
\newtheorem{prop}[thm]{Proposition}
\newtheorem{rem}[thm]{Remark}
\newcommand{\cref}[1]{Corollary~$\ref{#1}$}
\newcommand{\lref}[1]{Lemma~$\ref{#1}$}
\newcommand{\pref}[1]{Proposition~$\ref{#1}$}
\newcommand{\rref}[1]{Remark~$\ref{#1}$}
\newcommand{\tref}[1]{Theorem~$\ref{#1}$}
\newcommand{\Tref}[1]{Table~$\ref{#1}$}
\newcommand{\eref}[1]{$(\ref{e#1})$}
\newcommand{\secref}[1]{Section~$\ref{#1}$}

\renewcommand{\leq}{\leqslant}
\renewcommand{\ge}{\geqslant}
\renewcommand{\le}{\leqslant}
\renewcommand\emptyset{\varnothing}

\begin{document}

\title[Maximally nonassociative quasigroups]
{On the number of quadratic orthomorphisms that\\ produce
maximally nonassociative quasigroups}
\author{Ale\v s Dr\'apal}
\author{Ian M. Wanless}

\address{Department of Mathematics \\ Charles University
\\ Sokolovsk\'a 83 \\ 186
75 Praha 8, Czech Republic}
\address{School of Mathematics \\
Monash University \\
Clayton Vic 3800\\
Australia}
\email{drapal@karlin.mff.cuni.cz}
\email {ian.wanless@monash.edu}

\begin{abstract}
Let $q$ be an odd prime power and suppose that $a,b\in \F_q$ are such that 
$ab$ and $(1{-}a)(1{-}b)$ are nonzero squares. 
Let $Q_{a,b} = (\F_q,*)$ be the quasigroup in which the
operation is defined by $u*v=u+a(v{-}u)$ if $v-u$ is a square, and
$u*v=u+b(v{-}u)$ is $v-u$ is a nonsquare. This quasigroup 
is called \emph{maximally nonassociative} if it 
satisfies $x*(y*z) = (x*y)*z$ $\Leftrightarrow$ $x=y=z$. Denote
by $\sigma(q)$ the number of $(a,b)$ for which $Q_{a,b}$ is
maximally nonassociative.
We show that there exist constants $\alpha \approx 0.02908$ and $\beta
\approx 0.01259$ such that if $q\equiv 1 \bmod 4$, then
$\lim \sigma(q)/q^2 = \alpha$, and if $q \equiv 3 \bmod 4$,
then $\lim \sigma(q)/q^2 = \beta$. 
\end{abstract}

\maketitle

\section{Introduction}\label{1}

The existence of maximally nonassociative quasigroups was an open
question for quite a long time \cite{kep,gh,dv1}. In 2018 a maximally
nonassociative quasigroup of order nine was found \cite{dv3}, and that
was the first step to realise that Stein's nearfield
construction \cite{st} can be used to obtain maximally nonassociative
quasigroups of all orders $q^2$, where $q$ is an odd prime
power \cite{dl}. A recent result \cite{dw} constructs examples of all
orders with the exception of a handful of small cases and two sparse
subfamilies within the case $n\equiv 2\bmod 4$. The main construction
of \cite{dw} is based upon quadratic orthomorphisms and can be used
for all odd prime powers $q\ge 13$. However, it was left open how many
quadratic orthomorphisms can be used in the construction. We provide
an asymptotic answer to that question in this paper.

Throughout this paper $q$ is an odd prime power and
$\F = \F_q$ is a field of order $q$. For $a,b\in \F$ 
define a binary operation on $\F$ by
\begin{equation}\label{e11}
u*v = 
\begin {cases}
u+a(v-u)&\text{if $v-u$ is a square;}\\
u+b(v-u)&\text{if $v-u$ is a nonsquare.}
\end{cases}
\end{equation}
This operation yields a quasigroup 
if and only if
both $ab$ and $(1-a)(1-b)$ are squares, and both $a$ and $b$ are
distinct from $0$ and $1$, cf.~\cite{Eva18,cyclatom}. 
Denote by $\Sigma = \Sigma(\F)$ the
set of all such $(a,b) \in \F\times \F$ for which $a\ne b$. 

For each $(a,b) \in \Sigma$ denote the quasigroup $(\F,*)$
by $Q_{a,b} = Q_{a,b}(\F)$. A quasigroup $(Q,*)$ is said to be
\emph{maximally nonassociative} if 
\begin{equation}\label{e12}
(u*v)*w = u*(v*w)\ \Longrightarrow \ u=v=w
\end{equation}
holds for all $u,v,w \in Q$. By \cite{kep}, a maximally nonassociative 
quasigroup
has to be idempotent (i.e., $u*u=u$ for all $u\in Q$). Hence in
a maximally nonassociative quasigroup the converse implication  
to \eref{12} holds as well.

If $a=b\in \F\setminus\{0,1\}$, then \eref{11} defines a quasigroup
in which $u*(v*u) = (u*v)*u$ for all $u,v \in \F$. This means that
such a quasigroup is never maximally nonassociative. If $q\ge 13$,
then there always exists $(a,b) \in \Sigma(\F_q)$ such that
$Q_{a,b}$ is maximally nonassociative \cite{dw}. This paper is concerned
with the density of such $(a,b)$. Our main result is as follows:

\begin{thm}\label{11}
For an odd prime power $q$ denote by $\sigma(q)$ the number
of $(a,b) \in \Sigma(\F_q)$ for which $Q_{a,b}$ is maximally
nonassociative. Then 
\begin{equation}\label{e13}
\lim_{q\to \infty}\frac{\sigma(q)}{q^2} =
\begin{cases} 
953\cdot2^{-15}\approx 0.02908&\text{for $q\equiv 1\bmod 4$,}\\
825\cdot2^{-16}\approx 0.01259&\text{for $q\equiv 3\bmod 4$.}
\end{cases}
\end{equation}
\end{thm}

As we show below, the set $\Sigma$ consists of $(q^2-8q+15)/4$
elements.  Hence a random choice of $(a,b) \in \Sigma$ yields a
maximally nonassociative quasigroup with probability $\approx 1/8.596$ if
$q\equiv 1\bmod 4$, and with probability $\approx 1/19.86$ if $q\equiv
3 \bmod 4$. This may have an important consequence for the cryptographic
application described in \cite{gh}. It means that a maximally nonassociative
quasigroup of a particular large order can be obtained in an acceptable
time by randomly generating pairs $(a,b)$ until one is found for which
$Q_{a,b}$ is maximally nonassociative.

An important ingredient in the proof of \tref{11} is the
transformation described in \pref{12}, and used in \cref{13} to
determine $|\Sigma|$.

Define $S = S(\F)$ as the set of all $(x,y)\in \F\times \F$ such that
both $x$ and $y$ are squares, $x\ne y$ and $\{0,1\}\cap\{x,y\}=\emptyset$.

\begin{prop}\label{12}
For each $(a,b)\in \Sigma$ there exists exactly one $(x,y)\in S$
such that
\begin{equation}\label{e14}
a=\frac{x(1{-}y)}{x{-}y}, \quad b = \frac{1{-}y}{x{-}y}, \quad
1{-}a = \frac{y(1{-}x)}{y{-}x} \text{ \, and \, }
1{-}b=\frac{1{-}x}{y{-}x}.
\end{equation}
The mapping
\[
\Psi\colon \Sigma\to S, \quad (a,b) \mapsto \left ( \frac ab,
\frac{1{-}a}{1{-}b}\right)\]
is a bijection. If $(x,y)\in S$, then $\Psi\m ((x,y)) = (a,b)$
if and only if \eref{14} holds.
\end{prop}
\begin{proof} If $x,y,a,b\in \F$ satisfy $x\ne y$, 
$a =x(1{-}y)/(x{-}y)$ and
$b= (1{-}y)/(x{-}y)$, then 
\begin{equation} \label{e15}
\text{$1{-}a = y(1{-}x)/(y{-}x)$ and
$1{-}b = (1{-}x)/(y{-}x)$.}
\end{equation} 
Define\begin{equation*} 
\Phi\colon S\to \F\times \F, \quad (x,y)\mapsto \left ( \frac{x(1{-}y)}
{x{-}y}, \frac{1{-}y}{x{-}y}\right ).
\end{equation*}
Suppose that $(x,y) \in S$ and set $b = (1{-}y)/(x{-}y)$. Then
$b\ne0$ as $y\ne 1$, and $b\ne 1$ since $x\ne 1$. Put $a = xb$.
Then $a\ne 0$ since $b\ne 0$ and $x\ne 0$, and $a\ne b$ since
$x\ne 1$. Furthermore, $a\ne 1$ since $y\ne 0$ and $x\ne 1$.
Since $a=xb$, $ab=xb^2$ is a square. By \eref{15}, $1{-}a = y(1{-}b)$.
Hence $(1{-}a)(1{-}b) = y(1{-}b)^2$ is a square too. This verifies that
$\Phi$ may be considered as a mapping  $S\to \Sigma$.

Assume $(a,b) \in \Sigma$. By definition, $\Psi((a,b))=(x,y)$, where
$x=a/b$ and $y= (1{-}a)/(1{-}b)$. We have $x\notin \{0,1\}$ since
$a\ne 0$ and $a\ne b$. Similarly, $y\notin \{0,1\}$. Furthermore, 
$x\ne y$ since $x=y$ implies
$a = b$. Thus $(x,y)\in S$. By straightforward verification,
$\Psi\Phi = \id_S$ and $\Phi\Psi = \id_\Sigma$.
\end{proof}

\begin{cor}\label{13}
$|\Sigma(\F_q)| = |S(\F_q)| = (q^2-8q+15)/4$.
\end{cor}
\begin{proof} By \pref{12}, $|\Sigma| = |S|$. By the definition,
$S$ contains $((q-3)/2)^2 - (q-3)/2$ elements. 
\end{proof}

The definition of $Q_{a,b}$ follows the established way of defining
a quasigroup by means of an orthomorphism, say $\orth$, of an
abelian group $(G,+)$. Here, $\orth$ is said to be an \emph{orthomorphism}
of $(G,+)$ if it permutes $G$ and if the mapping $x\mapsto \psi(x)-x$
permutes $G$ as well. If $\orth$ is an orthomorphism, then
$x*y = x+\psi(y-x)$ is always a quasigroup. A \emph{quadratic orthomorphism}
$\orth=\orth_{a,b}$ is defined for each $(a,b)\in \Sigma(\F_q)$ by 
\begin{equation}\label{e16}
\orth(u) = \begin{cases} au \quad \text{\,if $u$ is a square;}\\
bu \quad \text{\ if $u$ is a nonsquare.}\end{cases}
\end{equation}
The definition \eref{11} of the quasigroup $Q_{a,b}$ thus
fits the general scheme. See~\cite{Eva18,diagcyc} for more 
information on quasigroups defined by means of orthomorphisms.

The maximal nonassociativity of $Q_{a,b}$ can be expressed via
the \emph{Associativity Equation:}
\begin{equation} \label{e17}
\orth(\orth(u)-v) = \orth(-v) + \orth(u-v-\orth(-v))
\end{equation}

\begin{prop}\label{14}
For $(a,b)\in \Sigma$ put $\orth = \orth_{a,b}$.
An ordered pair $(u,v)\in \F^2$ fulfils the Associativity Equation 
\eref{17} if and only
if $v*(0\,*u) = (v*\,0)*u$. Furthermore,
\begin{equation}\label{e18}
u-v-\orth(-v) = u - (v*0) \quad\text{and}\quad \orth(u) - v = (0*u)-v.
\end{equation}
If $(u,v)\ne (0,0)$ fulfils \eref{17}, then none of
$u$, $v$, $u-v-\orth(-v)$ and $\orth(u)-v$ vanishes,
and $(c^2u,c^2v)$ fulfils \eref{17} too, for any $c\in \F$.

The quasigroup $Q_{a,b}$ if maximally nonassociative if and only
if $(u,v)=(0,0)$ is the only solution to \eref{17}. 
\end{prop}
\begin{proof} This is a restatement of Lemmas~1.3 and~3.1 from
\cite{dw}. A sketch of the proof follows,
in order to make this paper self-contained. 
Since $u \mapsto z + u$ is an automorphism of $Q=Q_{a,b}$ for each
$z\in \F$, the
maximal nonassociativity is equivalent to having no $(u,v)\ne (0,0)$
such that $u*(0*v) = (u*0)*v$. This turns into \eref{17} 
by invoking the formula $u*v = u+\orth(v-u)$.
Since $x\mapsto c^2x$
is an automorphism of $Q$ for each $c\in \F$, $c\ne 0$, the 
Associativity Equation holds for $(u,v)$ if and only if it holds
for $(c^2u,c^2v)$. For the rest it suffices to observe
that in an idempotent quasigroup $u*(v*w) =(u*v) *w $ implies $u=v=w$
if $u=v$ or $u=v*w$ or $v=w$ or $u*v = w$. 
\end{proof}

For $(a,b)\in \Sigma$ denote by $E(a,b)$ the set of $(u,v)\ne (0,0)$
that satisfy the Associativity Equation \eref{17}. By \pref{14},
$Q_{a,b}$ is maximally nonassociative if and only if $E(a,b)= \emptyset$.
The number of such $(a,b)$ may be obtained indirectly by counting
the number of $(a,b)\in \Sigma$ for which $E(a,b)\ne \emptyset$. 
To this end, we will partition $E(a,b) = \bigcup E_{ij}^{rs}(a,b)$,
where $i,j,r,s\in \{0,1\}$. To determine
to which part an element $(u,v)\in E(a,b)$ belongs,
the following rule is used:
\begin{gather*} 
\begin{aligned}
i=0\quad  &\Longleftrightarrow \quad u \ \text {is a square};\\
j=0\quad &\Longleftrightarrow \quad -v \ \text {is a square};\\
r=0\quad &\Longleftrightarrow 
\quad \psi_{a,b}(u)-v \ \text {is a square, and};\\
s=0\quad &\Longleftrightarrow \quad u-v-\psi_{a,b}(-v) \ \text {is a square}.  
\end{aligned}
\end{gather*}
Thus, if one of the elements $u$, $-v$, $\psi_{a,b}(u){-}v$
and $u{-}v{-}\psi_{a,b}(-v)$ is a nonsquare, then the respective value
of $i$, $j$, $r$ or $s$ is set to $1$. For each $(u,v)\in E(a,b)$ there
hence exists exactly one quadruple $(i,j,r,s)$ such that
$(u,v) \in E_{ij}^{rs}(a,b)$, giving us the desired partition. 
We will also work with sets 
\begin{equation*} 
\Sigma_{ij}^{rs} = \{(a,b)\in \Sigma: E_{ij}^{rs}(a,b)\ne \emptyset\},
\end{equation*}
where $i,j,r,s\in \{0,1\}$.
The next observation directly follows from the definition 
of the sets
$\Sigma_{ij}^{rs}$. It is recorded here for the sake of later reference.

\begin{prop}\label{15}
Suppose that $(a,b)\in \Sigma=\Sigma(\F_q)$, for an odd
prime power $q>1$. The quasigroup $Q_{a,b}$ is maximally nonassociative
if and only if $(a,b)\notin \bigcup \Sigma_{ij}^{rs}$.
\end{prop}

If it is assumed that $(u,v) \in E_{ij}^{rs}(a,b)$, then 
the Associativity Equation \eref{17} can be turned into a linear 
equation in unknowns $u$ and $v$ since each occurrence of $\psi$ can
be interpreted by means of \eref{16}. The list of these linear equations
can be found in \cite{dw}. Their derivation is relatively short
and is partly repeated in Lemmas~\ref{24}--\ref{27}. The approach used here
differs from that of \cite{dw} in two aspects. The symmetries induced
by opposite quasigroups and by automorphisms $Q_{a,b}\cong Q_{b,a}$
are used more extensively here, and characterisations
of $\Sigma_{ij}^{rs}$ are immediately transformed 
into characterisations of 
\begin{equation}\label{e110}
S_{ij}^{rs} = \Psi(\Sigma_{ij}^{rs}).
\end{equation}
As will turn out, sets $S_{ij}^{rs}$ can be described by a requirement
that several polynomials in $x$ and $y$ are either squares or nonsquares.
Estimates of $|S_{ij}^{rs}|$ can be thus obtained by means of the Weil bound
(as formulated, say, in \cite[Theorem 6.22]{Evans}). 
We shall not be using the Weil bound directly, but via \tref{16} below, 
a straightforward consequence from \cite[Theorem~1.4]{dw}.
Applications of \tref{16} to the intersections of 
sets $S_{ij}^{rs}$, with symmetries taken into account, yield, after
a number of computations, the asymptotic results stated in \tref{11}. 

Say that a list of polynomials $p_1,\dots,p_k$ in one variable,
with coefficients in $\F$,
is \emph{square-free}
if there exists no sequence $1\le i_1<\dots<i_r \le k$ such that
$r\ge 1$ and $p_{i_1}\cdots p_{i_k}$ is a square (as a polynomial
with coefficients in the algebraic closure $\bar \F$ of $\F$).
Define $\chi\colon \F\to \{\pm1,0\}$ to be the
quadratic character extended by $\chi(0)=0$. 

\begin{thm}\label{16}\label{t:Weil}
Let $p_1,\dots,p_k\in \F[x]$ be a square-free list of polynomials of
degree $d_i\ge1$, and let $\eps_1,\dots,\eps_k\in\{-1,1\}$. Denote by
$N$ the number of all $\alpha \in \F$ such that
$\chi(p_i(\alpha))= \eps_i$, for $1\le i \le k$.
Then
$$|N-2^{-k}q|<(\sqrt q+1)D/2-\sqrt{q}(1-2^{-k})<(\sqrt q+1)D/2
$$
where $D=\sum_i d_i$.
\end{thm}

The purpose of \secref{2} is to describe each of the sets $S_{ij}^{rs}$ 
by a list of polynomials $p(x,y)$ such that the presence
of $(x,y)\in S$ in $S_{ij}^{rs}$ depends upon $p(x,y)$ being
a square or nonsquare.
\tref{210} gives such a description for $q=|\F|\equiv 1 \bmod 4$,
and \tref{211} for $q\equiv 3 \bmod 4$. \secref{3} contains 
auxiliary results that make applications of \tref{16} possible.
Note that \tref{16} is concerned with polynomials in only one
variable. To use it, one of the variables, say $y$, has to be
fixed. If $y=c$, and $p_1(x,y),\dots,p_k(x,y)$ are the polynomials
occurring in Theorems~\ref{210} and~\ref{211}, then \tref{16}
may be used without further specifications 
only for those $c$ for which $p_1(x,c),\dots,p_k(x,c)$
is a square-free list. The purpose of \secref{3} is to show
that this is true for nearly all $c$, and that the number of possible
exceptional values of $c$ is very small. \secref{4} provides 
the estimate of $S\setminus \bigcup S_{ij}^{rs}$ for $q\equiv 3\bmod4$,
and \secref{5} for $q\equiv 1 \bmod 4$, cf.~Theorems~\ref{46} and \ref{53}.
\secref{6} consists of concluding remarks.

\section{Quadratic residues and the Associativity Equation}\label{2}

Let $Q^{op}_{a,b}$ denote the opposite quasigroup of $Q_{a,b}$, namely the
quasigroup satisfying $Q^{op}_{a,b}(u,v)=Q_{a,b}(v,u)$ for all $u,v$.
The following facts are well known \cite{Eva18,cyclatom} and easy to verify:

\begin{lem}\label{21}
If $(a,b)\in \Sigma$, then 
\begin{enumerate}
\item[(i)]
$u\mapsto u\nsq$ is an isomorphism
$Q_{a,b}\cong Q_{b,a}$, for every nonsquare
$\nsq \in \F$;
\item[(ii)] $Q^{op}_{a,b} = Q_{1-a,1-b}$ if $q\equiv 1\bmod 4$,
and $Q^{op}_{a,b} = Q_{1-b,1-a}$ if $q\equiv 3\bmod 4$.
\end{enumerate}
\end{lem}

An alternative way to express that $q\equiv 1 \bmod 4$ is to say
that $-1$ is a square. If $\bar *$ denotes the operation of the
opposite quasigroup, then  $(v\,\bar *\, 0)\,\bar *\, u = 
v\,\bar *\, (0 \,\bar*\, u)$ holds in $Q_{a,b}^{op}$ if and only if
$u*(0*v) = (u*0)*v$. Hence $(u,v)\in E(a,b)$ if and only if
$(v,u)\in E(a',b')$, where $(a',b')=(1{-}a,1{-}b)$ if $-1$ is a square,
and $(a',b')= (1{-}b,1{-}a)$ if $-1$ is a nonsquare, by part (ii)
of \lref{21}. Similarly $(u,v)\in E(a,b)$ $\Leftrightarrow$
$(\nsq u,\nsq v)\in E(b,a)$. 

Working out these connections with respect to being square or nonsquare
yields the following statement. It appears without a proof since it 
coincides with Lemmas~3.2 and~3.3 of \cite{dw} and since the proof is 
straightforward. 

\begin{lem}\label{22}
Assume $(a,b)\in \Sigma$ and $i,j,r,s\in \{0,1\}$.
Then 
\begin{align} \label{e21}
(u,v)\in E_{ij}^{rs}(a,b) \ &\Longleftrightarrow \
(\nsq u,\nsq v) \in E_{1-i,1-j}^{1-r,1-s}(b,a); \\ \label{e22}
(u,v)\in E_{ij}^{rs}(a,b) \ &\Longleftrightarrow \
(v,u) \in E_{ji}^{sr}(1{-}a,1{-}b) \text{ if $-1$ is a square; and} \\
\label{e23}
(u,v)\in E_{ij}^{rs}(a,b) \ &\Longleftrightarrow \
(v,u) \in E_{1-j,1-i}^{1-s,1-r}(1{-}b,1{-}a) \text{ if $-1$ is a nonsquare.}
\end{align}
\end{lem}

\begin{prop}\label{23}
Both of the mappings $(x,y)\mapsto (y,x)$ and $(x,y)\mapsto (x\m,y\m)$
permute the set $S = S(\F)$. If $i,j,r,s\in \{0,1\}$, then 
\begin{equation*}
(x,y) \in S_{ij}^{rs}\ \Longleftrightarrow\ (y,x) \in S_{ji}^{sr}
\ \Longleftrightarrow\  (x\m,y\m)\in S_{1-i,1-j}^{1-r,1-s}.
\end{equation*}
\end{prop}

\begin{proof}
By definition, $(x,y)\in S$ if and only if $x$ and $y$ are both squares,
$x\ne y$, and $\{x,y\}\cap\{0,1\}=\emptyset$. These
properties are retained both by the switch $(x,y)\mapsto (y,x)$
and by the inversion $(x,y)\mapsto (x\m,y\m)$. These mappings thus
permute $S$. 

Let $(a.b)\in \Sigma$ be such that $\Psi((a,b)) = (x,y)$. Then
$x=a/b$, $y = (1{-}a)/(1{-}b)$. Hence $\Psi((b,a)) = (x\m,y\m)$
and $\Psi((1{-}a,1{-}b)) = (y,x)$. For the proof we thus need to show
that 
\begin{equation*}
(a,b) \in \Sigma_{ij}^{rs}\ \Longleftrightarrow\ (1{-}a,1{-}b) \in 
\Sigma_{ji}^{sr}
\ \Longleftrightarrow\  (b,a)\in \Sigma_{1-i,1-j}^{1-r,1-s}.
\end{equation*}
Suppose that $(a,b) \in \Sigma_{ij}^{rs}$, i.e., that there exists
$(u,v) \in E_{ij}^{rs}(a,b)$. If $-1$ is a square, then $(v,u)\in
E_{ji}^{sr}(1{-}a,1{-}b)$ by \eref{22}. If $-1$ is a nonsquare,
then $(\nsq v,\nsq u) \in E_{ji}^{sr}(1{-}a,1{-}b)$, by \eref{23}
and \eref{21}. Thus $(1{-}a,1{-}b)\in \Sigma_{ji}^{sr}$ in both cases.
We also have $(\nsq u,\nsq v) \in E_{1-i,1-j}^{1-r,1-s}(b,a)$, by \eref{21}.
Hence $(b,a)\in \Sigma_{1-i,1-j}^{1-r,1-s}$.
\end{proof}

To determine all of the sets $S_{ij}^{rs}$ it thus suffices to know the
sets
\begin{equation}\label{e24}
S_{00}^{00},\ S_{00}^{01},\ S_{00}^{11},\ S_{01}^{00},\ S_{01}^{01}
\text{ and } S_{01}^{10}.
\end{equation}
We next determine these sets via a sequence of Lemmas.

\begin{lem}\label{24}
If $-1$ is a square, then $S_{01}^{00}=S_{01}^{10}=\emptyset$, while
\begin{align*}
(x,y)\in S_{00}^{00} \ &\Longleftrightarrow \
\text{$(1{-}x)(y{-}x)$ and $(1{-}y)(y{-}x)$ are squares; and}\\
(x,y)\in S_{00}^{11} \ &\Longleftrightarrow \
\text{$(x^2y+xy-x^2-y^2)(y{-}x)$ and $(xy^2+xy-x^2-y^2)(y{-}x)$}\\
&\phantom{\Longleftrightarrow \ } \ \text{are nonsquares.}
\end{align*}
\end{lem}

\begin{proof} We assume that $-1$ is a square.
If $(u,v)\in E_{00}^{00}(a,b)$, then
the Associativity Equation attains the form $a(au-v)=-av+a(u-v+av)$,
and that is the same as $(1{-}a)(u{-}v)=0$. Since $1{-}a\ne 0$,
and since $u$ is assumed to be square, the set $E_{00}^{00}(a,b)$
is nonempty if and only if it contains $(1,1)$, by \pref{14}.
This takes place if and only if $1{-}a$ and $a$ are squares.
Suppose that $(x,y) = \Psi((a,b))$.
Then $a=x(1{-}y)/(x{-}y)$ is a square if and only if
$(1{-}y)(y{-}x)$ is a square, and $1{-}a =  y(1{-}x)/(y{-}x)$ is a square
if and only if $(1{-}x)(y{-}x)$ is a square.

If $(u,v) \in E_{00}^{11}(a,b)$, then $b(au-v)=-av+b(u-v+av)$
yields $ub(a{-}1)=a(b{-}1)v$, where both $u$ and $v$ are squares. 
Thus $(u,v)$ is a solution if and only if $(1,b(a{-}1)/a(b{-}1))$
is a solution. Since $v=b(a{-}1)/a(b{-}1)$ is always a square,
the conditions for the existence of the solution are that
$a-v$ and $1-(1{-}a)v$ are nonsquares. If $(x,y) = \Psi((a,b))$,
then $v = y/x$, $a-v = (x^2-x^2y-yx+y^2)/x(x-y)$ and
$1-(1{-}a)v=(xy-x^2-y^2+y^2x)/x(y{-}x)$.

If $(u,v) \in E_{01}^{00}(a,b)$, then $a(au-v)=-bv + a(u-v+bv)$
and $a(a-1) u = b(a-1)v$. This implies that $uv$ is a square.
However, the assumption $(u,v) \in E_{01}^{00}(a,b)$ implies that
$u$ is a square and $-v$ is a nonsquare. Thus $uv$ should
be both a square and a nonsquare, which is a contradiction.
If $(u,v) \in E_{01}^{10}(a,b)$, then $b(au-v)=-bv + a(u-v+bv)$,
and that gives $u=v$, a contradiction again.  
\end{proof}

\begin{lem}\label{25}
If $-1$ is a nonsquare, then $S_{00}^{00}=S_{00}^{11}=\emptyset$, while
\begin{align*}
(x,y) \in S_{01}^{10} &\ \Longleftrightarrow \ (x,y) \in S_{10}^{01};\\
&\ \Longleftrightarrow \ \text{$(1{-}y)(x{-}y)$ and $(1{-}x)(y{-}x)$
are squares; and} \\
(x,y) \in S_{01}^{00} &\ \Longleftrightarrow \ \text{$(x{-}1)(y{-}x)$
and $(x^2{-}2x{+}y)(y{-}x)$ are squares.}
\end{align*}
\end{lem}

\begin{proof} We assume that $-1$ is a nonsquare.
If $E_{00}^{00}(a,b)\ne \emptyset$, then $(1,1)\in E_{00}^{00}
(a,b)$, by the same argument as in the proof of \lref{24}. 
However, $(1,1)$ cannot belong to $E_{00}^{00}(a,b)$ since $-1$ is
a nonsquare. Similarly, $E_{00}^{11}(a,b) = \emptyset$ since
$-b(a{-}1)/a(b{-}1)$ is a nonsquare.

Suppose that $(u,v) \in E_{01}^{00}(a,b)$. Then \eref{17} implies $au = bv$. 
Hence $(a,b)\in \Sigma_{01}^{00}$ $\Leftrightarrow$ $(1,a/b)\in E_{01}^{00}(a,b)$.
The latter takes place if and only if $a-a/b$ and $1-(1{-}b)a/b$ are squares.
Let $(x,y) = \Psi((a,b))$. Then $a-a/b = x((1{-}y)/(x{-}y)-1) =
x(1{-}x)/(x{-}y)$ and $1-(1{-}b)a/b=1-x(1-x)/(y{-}x)=(x^2{-}2x{+}y)/(y{-}x)$.

Let $(u,v)\in E_{01}^{10}(a,b)$. Then $u=v$ by \eref{17}. Hence
$(a,b)\in \Sigma_{01}^{10}$ if and only if $(1,1)\in E_{01}^{10}(a,b)$.
The latter is true if and only if $a-1$ is a nonsquare and $b$
is a square. If $(x,y) = \Psi((a,b))$, then this means that
$(x{-}1)(y{-}x)$ is a nonsquare and $(1{-}y)(x{-}y)$ is a square.
The symmetry of these conditions shows that
$(x,y)\in S_{01}^{10}\Leftrightarrow(y,x)\in S_{01}^{10}$. Hence 
$S_{01}^{10}=S_{10}^{01}$, by \pref{23}.
\end{proof}

\begin{lem}\label{26}
Assume that $(x,y) \in S$. Then $(x,y) \in S_{00}^{01}$ if and only if
$-xy-y+x$ and $(-x^2y+x^2+y^2-xy)(x{-}y)$ are squares,
and $(1{-}y)(x{-}y)$ is a nonsquare. 
\end{lem}
\begin{proof} In this case the Associativity Equation is equal to
$a(au-v)=-av+b(u-v+av)$, and that is the same as $(a^2-b)u=(ab-b)v$.
Therefore $(a,b)\in\Sigma_{00}^{01}$ if and only if
$(1,(a^2{-}b)/b(a{-}1))\in E_{00}^{01}(a,b)$. If $(x,y) = \Psi((a,b))$,
then $(a^2{-}b)(x{-}y)^2 = x^2(1{-}y)^2-(1{-}y)(x{-}y)= (1-y)(x^2-x^2y-x+y)
=(1{-}y)(1{-}x)(y+xy-x)$,
and $b(a{-}1)(x{-}y)^2 = (1{-}y)y(1{-}x)$. Hence 
$v=(a^2-b)/b(a-1)=(y+xy-x)/y$, showing that $y+xy-x$ is a square.
It follows that
$a-v 
=(-x^2y+x^2+y^2-xy)/(x-y)y$ 
and $(1-(1{-}a)v)(y{-}x) 
=x^2(y{-}1)$. Thus $(1{-}y)(x{-}y)$ has to be
a nonsquare. 
\end{proof}

\begin{lem}\label{27}
Assume that $(x,y)\in S$. 
\begin{enumerate}
\item[(i)] If $y{+}1{-}x = 0 = x^2{-}x{-}1$ and $q>43$, 
then $(x,y) \in S_{01}^{01}$.
\item[(ii)] If $y{+}1{-}x \ne 0$ or $x^2{-}x{-}1\ne 0$,
then $(x,y)\in S_{01}^{01}$ if and only if both
$(y{+}xy{-}x)(x{-}y{-}1)$ and $(y{-}2x{+}x^2)(x{-}y)(x{-}y{-}1)$
are nonsquares,
while $(2xy{-}y^2{-}x)(x{-}y)(x{-}y{-}1)$
is a square.
\end{enumerate}
\end{lem}

\begin{proof}
In this case the Associativity Equation yields $a(au-v) = -bv + b(u -
(1{-}b)v)$. That is equivalent to $(a^2-b)u = (b^2-2b+a)v$. If there
exists a solution $(u,v)\in E_{01}^{01}(a,b)$, and one of the elements
$a^2-b$ and $b^2-2b+a$ is equal to zero, then the other has to vanish
as well. Assume that $(x,y)=\Psi((a,b))$. Then $a^2-b=0$
if and only if $0 = x^2(1-y)^2-(1-y)(x-y) = (1-y)(-x^2y+x^2-x+y)
= (1-y)(1-x)(y+xy-x)$, and $b^2-2b+a = (1-b)^2-(1-a)=
0$ if and only if $0 = (1-x)^2-y(1-x)(y-x) = (1-x)(1-x-y^2+xy)=
(1-x)(1-y)(y-x+1)$. If $y=x-1$, then $y+xy-x=x^2-x-1$.

Computations above show that 
\begin{equation*}
a^2-b = \frac{(1{-}y)(1{-}x)(xy{-}x{+}y)}{(x{-}y)^2}
\text{ \, and \,} b^2-2b+a = 
\frac{(1{-}y)(1{-}x)(y{-}x{+}1)}{(x{-}y)^2}.
\end{equation*}

Suppose now that at least one of $x^2{-}x{-}1$ and $y{-}x{+}1$ does not vanish.
If $y{-}x{+}1=0$, then $E_{01}^{01}(a,b)=\emptyset$ 
and $(y{+}xy{-}x)(x{-}y{-}1)=0$, which 
is a square. Hence $y{-}x{+}1\ne 0$ may be assumed. 
That implies $b^2-2b+a\ne 0$. From the Associativity Equation
it then follows that  $(a,b)\in \Sigma_{01}^{01}$ if and only if
$(1,v)\in E_{01}^{01}(a,b)$, where $v=(a^2-b)/(b^2-2b+a)$. Now,
\begingroup
\begin{align*}
-v &= \frac{b-a^2}{b^2-2b+a} = 
\frac{y{+}xy{-}x}{x{-}y{-}1}, \\
1-(1{-}b)v&= \frac{(x{-}y{-}1)(y{-}x)+(1{-}x)(y{+}xy{-}x)}{(x{-}y{-}1)(y{-}x)}
=\frac{y(x^2{-}2x{+}y)}{(x{-}y{-}1)(x{-}y)}, \text{ and}\\
a{-}v &= \frac{x(1{-}y)(x{-}y{-}1) + (x{-}y)(y{+}xy{-}x)}{(x{-}y)(x{-}y{-}1)}
= \frac{2xy{-}y^2{-}x}{(x{-}y)(x{-}y{-}1)}.
\end{align*}
\endgroup

It remains to prove that $E_{01}^{01}(a,b)$ is nearly always nonempty
if $a^2-b=b^2-2b+a=0$.
Let the latter be true. Then $b^2-2b+a=a^4-2a^2+a = a(a{-}1)(a^2{+}a{-}1)$.
Thus $a^2{+}a{-}1=0$. A pair $(1,v)$ is a solution to the Associativity 
Equation if $-v$ is a nonsquare, $1+(1{-}b)(-v)$ is a nonsquare, and
$a-v$ is a square. Put $p_1(t) = t$, $p_2(t) = 1+(1{-}b)t=1+(1{-}a^2)t$,
and $p_3(t) = a+t$. A solution $(1,v)$ exists if there exists $\gamma=-v\in \F$
such that $\chi(p_1(\gamma))=\chi(p_2(\gamma))=-1$ and 
$\chi(p_3(\gamma)) = 1$. Polynomials $p_2$ and $p_3$ have a common
root if and only if $0=1-a+a^3$. 
If this is true, then 
$0=a^2+a^3 = a^2(1+a)$. This implies $a=-1$ and
$0 = (-1)^2+(-1)-1 =-1$, a contradiction. The list of polynomials $p_1$, 
$p_2$, $p_3$ is therefore square-free.  \tref{16} guarantees the 
existence of $\gamma$
if $0<q/8 -(\sqrt q+1)(3/2)+\sqrt q(1-1/8)=q/8-5\sqrt q/8-3/2$. 
This is true for each prime power $q\ge 47$.
\end{proof}

\begin{rem}\label{28}
Lemmas~$\ref{24}$--$\ref{27}$ cover all sets $S_{ij}^{rs}$ that are listed
in \eref{24}. Up to the exceptions discussed in \rref{29},
each of these sets is either empty, or  is described by a 
list of polynomials, say $p_1,\dots,p_k\in \F[x,y]$, $k\in \{2,3\}$,
and elements $\eps_h\in \{-1,1\}$, such that
$(\xi,\eta)\in S$ belongs to $S_{ij}^{rs}$ if and only if 
$\chi(p_h(\xi,\eta))=\eps_h$, for $1\le h \le k$. This is because
the polynomials
$p_h(x,y)$ have been determined in all cases in such a way 
that if $p_h(\xi,\eta)=0$ and $(\xi,\eta) = \Psi((a,b))$,
then there is no $(u,v)\in E_{ij}^{rs}(a,b)$. Indeed if $(u,v)$ were
such a solution, then $u$ or $v$ or $u-v-\orth_{a,b}(-v)$  or
$\orth_{a,b}(u)-v$ would be equal to zero, and that is impossible,
by \pref{14}.
\end{rem}

Note that $(\xi,\eta)$ was used in \rref{28} to emphasise the
distinction between elements of $S$ and formal variables $x$ and $y$.
In the remainder of the paper, elements of $S$ will again be denoted
by $(x,y)$. The context will always be clear.

\begin{rem}\label{29}
Sets $S_{01}^{01}$ and $S_{10}^{10}$ behave exceptionally in the sense
that the regular behaviour described in \rref{28} needs an assumption
that $y{+}1{-}x\ne 0$ or $x^2{-}x{-}1\ne 0$ (for the set
$S_{01}^{01}$), and that $x{+}1{-}y\ne 0$ or $y^2{-}y{-}1\ne 0$ (for
the set $S_{10}^{10}$). There are at most two pairs $(x,y) \in S$ such
that $y{+}1{-}x=0=x^2{-}x{-}1$ and at most two pairs $(x,y) \in S$
such that $x{+}1{-}y=0=y^2{-}y{-}1$. Hence assuming that
\begin{equation}\label{e25}
\text{$\bigl [y+1-x\ne 0$ or $x^2-x-1\ne 0\bigr ]$ and
$\bigl [x+1-y\ne 0$ or $y^2-y-1\ne 0\bigr ]$}
\end{equation}
causes no difficulty when estimating $\sigma(q)$.
If \eref{25} does not hold, then $(x,y) \in S_{01}^{01}\cup
S_{10}^{10}$ if $q\ge 47$, by point (i) of \lref{27}. In fact,
if \eref{25} does not hold, then $(x,y) \in 
\bigcup S_{ij}^{rs}$ for each $q\ge3 $, by \cite{dw} (cf.~the application
of \cite[Lemma~3.4]{dw} in the proof of \cite[Theorem~3.5]{dw}).
\end{rem}

For $p(x,y)\in \F[x,y]$ such that $x\nmid p(x,y)$ and $y\nmid p(x,y)$
define the \emph{reciprocal} polynomial 
$\hat p(x,y)$ as $x^ny^mp(x\m,y\m)$,  where $n$ and $m$
are the degree of the polynomial $p$ in the variables $x$ and $y$, respectively.
Note that if $(x,y)\in S$
then $\chi(\hat p(x,y))=\chi(x^ny^mp(x\m,y\m))=\chi(p(x\m,y\m))$
since $x$ and $y$ are squares.
Note also that $\hat{\hat p}(x,y)= p(x,y)$, $\widehat{1-x} = x-1$,
$\widehat{x-y}=y-x$ and $\widehat{x-1-y}=y-xy-x$. Set
\begin{gather}\label{e26}
\begin{aligned}
&f_1(x,y)=x^2{+}y^2{-}xy{-}x, \quad \quad &f_2(x,y) = y^2{+}x^2{-}xy{-}y,
\ \ \ \\
&f_3(x,y)=y^2x{+}xy{-}x^2{-}y^2\ \ \ \text{and}
&f_4(x,y)=x^2y{+}xy{-}x^2{-}y^2.
\end{aligned}
\end{gather}
Then $f_2(x,y) = f_1(y,x)$, $f_3(x,y) = -\hat{f_1}(x,y)$
and $f_4(x,y) = -\hat{f_1}(y,x)=-\hat{f_2}(x,y)=f_3(y,x)$.

A description of those sets $S_{ij}^{rs}$ that do not occur in \eref{24} can be 
derived from Lemmas~\ref{24}--\ref{27} by means of \pref{23}.
As an example consider sets $S_{00}^{10}$ and $S_{11}^{10}$.
By \lref{26}, $(x,y)\in S_{00}^{01}$ if $\chi(x{-}xy{-}y)=\chi(f_4(x,y)
(y{-}x)) = 1$ and $\chi((1{-}y)(x{-}y))=-1$. By \pref{23},
$(x,y) \in S_{00}^{10}$ if and only if $(y,x)\in S_{00}^{01}$,
i.e., if $\chi(y{-}xy{-}x)=\chi(f_3(x,y)(x{-}y))=1$ and $\chi
((1{-}x)(y{-}x))=-1$, and $(x,y)\in S_{11}^{10}$ if $(x\m,y\m)\in
S_{00}^{01}$, i.e., if $\chi(y{-}1{-}x)=\chi(f_2(x,y)(y{-}x))=1$
and $\chi((1{-}y)(x{-}y))=-1$.

Following this pattern a characterisation of all sets $S_{ij}^{rs}$
may be derived from Lemmas~\ref{24}--\ref{27} by means of \pref{23}.
This is done in Theorems~\ref{210} and~\ref{211}. Since the derivation
is straightforward, both of them are stated without a proof. Set
\begin{gather}\label{e27}
\begin{aligned}
&g_1(x,y)= x^2 + y -2x,&  &g_2(x,y)= y^2+x-2y,\\
&g_3(x,y)= x^2 + y -2xy& \text{and}\quad &g_4(x,y) = y^2+x-2xy.
\end{aligned}
\end{gather}
Note that $g_3(x,y)=\hat{g_1}(x,y)$, $g_4(x,y) = \hat{g_2}(x,y)=
g_3(y,x)$ and $g_2(x,y) = g_1(y,x)$.

\begin{thm}\label{210}
Assume that $q\equiv 1\bmod 4$ is a prime power, and that
$S = S(\F_q)$. Let $(x,y)\in S$ be such that \eref{25} holds.
The sets $S_{01}^{00}$, $S_{01}^{10}$,
$S_{01}^{11}$, $S_{10}^{00}$, $S_{10}^{01}$ and $S_{10}^{11}$
are empty, and $S_{11}^{11} = S_{00}^{00}$. 
Put $\eps = \chi(x{-}y)$. Then 
\begin{align*}
&(x,y)\in S_{00}^{00}& &\Longleftrightarrow& &
\chi(1{-}x)=\chi(1{-}y)=\eps;\\
&(x,y)\in S_{11}^{00}& &\Longleftrightarrow& &
\chi(f_1(x,y))=\chi(f_2(x,y)) = -\eps; \\
&(x,y)\in S_{00}^{11}& &\Longleftrightarrow& &
\chi(f_3(x,y))=\chi(f_4(x,y)) = -\eps; \\
&(x,y)\in S_{11}^{01}& &\Longleftrightarrow& &
\chi(1{-}x) = -\eps,\ \chi(y{+}1{-}x)=1 
\text{ and } \chi(f_1(x,y))=\eps; \\
&(x,y)\in S_{11}^{10}& &\Longleftrightarrow& &
\chi(1{-}y) = -\eps,\ \chi(x{+}1{-}y)=1 
\text{ and } \chi(f_2(x,y))=\eps; \\
&(x,y)\in S_{00}^{10}& &\Longleftrightarrow& &
\chi(1{-}x) = -\eps,\ \chi(x{+}xy{-}y)=1
\text{ and } \chi(f_3(x,y))=\eps; \\
&(x,y)\in S_{00}^{01}& &\Longleftrightarrow& &
\chi(1{-}y) = -\eps,\ \chi(y{+}xy{-}x)=1 
\text{ and } \chi(f_4(x,y))=\eps; \\
&(x,y) \in S_{01}^{01}& &\Longleftrightarrow& &
\chi(y{+}xy{-}x)=-\eta,\ \chi(g_1(x,y))=-\eta\eps
\text{ and } \chi(g_4(x,y)) = \eta\eps, \\
&&&&&\text{where }\eta = \chi(y{+}1{-}x); \text{ and} \\
&(x,y) \in S_{10}^{10}& &\Longleftrightarrow& &
\chi(x{+}xy{-}y)=-\eta,\ \chi(g_2(x,y))=-\eta\eps
\text{ and } \chi(g_3(x,y)) = \eta\eps, \\
&&&&&\text{where }\eta = \chi(x{+}1{-}y).
\end{align*}
\end{thm}

\begin{thm}\label{211}
Assume that $q\equiv 3\bmod 4$ is a prime power, and that
$S = S(\F_q)$. Let $(x,y)\in S$ be such that \eref{25} holds.
Sets $S_{00}^{00}$, $S_{00}^{11}$, $S_{11}^{00}$ and $S_{11}^{11}$
are empty, and $S_{10}^{01} = S_{01}^{10}$. The pair $(x,y)$
belongs to a set $S_{ij}^{rs}$ listed below if and only if all values
in the row of $S_{ij}^{rs}$ are nonzero squares.
\begin{align*}
S_{01}^{10}\colon&\quad (1{-}y)(x{-}y) \text{ and \,} (1{-}x)(y{-}x);\\
S_{01}^{00}\colon&\quad (1{-}x)(x{-}y) \text{ and \,}  g_1(x,y)(y{-}x); \\
S_{10}^{00}\colon&\quad (1{-}y)(y{-}x) \text{ and \,}  g_2(x,y)(x{-}y); \\
S_{10}^{11}\colon&\quad (1{-}x)(x{-}y) \text{ and \,}  g_3(x,y)(x{-}y); \\
S_{01}^{11}\colon&\quad (1{-}y)(y{-}x) \text{ and \,}  g_4(x,y)(y{-}x); \\
S_{11}^{01}\colon&\quad (1{-}x)(x{-}y),\ x{-}1{-}y
\text{ and \,} (x{-}y)f_1(x,y);\\
S_{11}^{10}\colon&\quad (1{-}y)(y{-}x),\ y{-}1{-}x
\text{ and \,} (y{-}x)f_2(x,y);\\
S_{00}^{10}\colon&\quad (1{-}x)(x{-}y),\ y{-}xy{-}x
\text{ and \,} (x{-}y)f_3(x,y);\\
S_{00}^{01}\colon&\quad (1{-}y)(y{-}x),\ x{-}xy{-}y
\text{ and \,} (y{-}x)f_4(x,y);\\
S_{01}^{01}\colon&\quad (x{-}xy{-}y)(x{-}1{-}y),\ 
g_1(x,y)(y{-}x)(x{-}1{-}y)\text{ and \,}
g_4(x,y)(y{-}x)(x{-}1{-}y); 
\\
S_{10}^{10}\colon&\quad (y{-}xy{-}x)(y{-}1{-}x),\ 
g_2(x,y)(x{-}y)(y{-}1{-}x)\text{ and \,}
g_3(x,y)(x{-}y)(y{-}1{-}x).
\end{align*}
\end{thm}

\bigskip

\section{Avoiding squares}\label{3}

Our goal is to estimate the size of the set $T=S\setminus\bigcup
S_{ij}^{rs}$.  Since \tref{16} requires polynomials in one variable,
to determine the size of $T$ it is necessary to proceed by determining
the sizes of slices $\{x\in \F:(x,c)\in T\}$, for each square
$c\notin \{0,1\}$.  As a convention, $p(x,c)$ will mean a polynomial
in one variable, i.e., an element of $\F[x]$, for every
$p(x,y)\in \F[x,y]$.

\tref{16} may be directly applied only when the product of the
polynomials involved is square-free.  Thus for
$p_1(x,y),\dots,p_k(x,y)\in \F[x,y]$ it is necessary to set aside
those $c\in \F$ for which $p_1(x,c),\dots,p_k(x,c)$ is \emph{not} a
square-free list of polynomials. An asymptotic estimate does not
depend upon the number of $c$ set aside if there are only a bounded
number of them. Hence a possible route is to express the discriminant
of $p_1(x,c)\cdots p_k(x,c)$ by means of computer algebra, and then set
aside those $c$ that make the discriminant equal to zero. The route
taken below is elementary and is not dependent upon computer. In this
way the number of $c$ to avoid is limited to 51.  This is a
consequence of the following statement, the proof of which is the goal
of this section.

\begin{thm}\label{31}
Let $\F$ be a field of characteristic different from $2$. 
The list of polynomials
\begin{gather}\label{e31}
\begin{aligned}
&\text{$x$, 
$x{-}1$, $x{-}c$, $x{-}1{-}c$, $x{+}1{-}c$, $(1{-}c)x-c$, $(1{+}c)x-c$,}
\\ &\text{$g_1(x,c)$, $g_2(x,c)$,
$g_3(x,c)$, $g_4(x,c)$, $f_1(x,c)$, $f_2(x,c)$,
$f_3(x,c)$, $f_4(x,c)$}
\end{aligned}
\end{gather}
is square-free if the following conditions hold:
\begin{align}
c&\notin \{-1,0,1,1/2,2\}; \label{e32}\\
c&\text{ is not a root of $x^2\pm x\pm1$;}
\label{e33}\\
c&\text{ is not a root of $x^2-3x+1$;}\label{e34}\\
c&\notin \{-1/3,-3,2/3,3/2,1/3,3,4/3,3/4\} \text{ if } \chr(F)\ne 3;\label{e35}\\
c&\text{ is a root of neither $x^2{-}3x{+}3$ nor $3x^2{-}3x{+}1$;} \label{e36} \\
c&\text{ is a root of neither $x^3{+}x^2{-}1$ nor $x^3{-}x{-}1$;}\label{e37}\\
c&\text{ is not a root of $x^2{+}1$;}\label{e38}\\
c&\text{ is a root of neither $x^2{-}2x{+}2$ nor $2x^2{-}2x{+}1$; }\label{e39}\\
c&\text{ is a root of neither $x^3{-}x^2{+}2x{-}1$ nor $x^3{-}2x^2{+}x{-}1$; and} \label{e310}\\
c&\text{ is a root of neither $x^3{-}2x^2{+}3x{-}1$ nor $x^3{-}3x^2{+}2x{-}1$.} \label{e311}
\end{align}
\end{thm}

The proof requires a number of steps. As an auxiliary notion,
we call a list of polynomials $p_1(x,y),\dots,p_k(x,y)\in
\F[x,y]$ \emph{reciprocally
closed} if for each $i\in \{1,\dots,k\}$ both $x\nmid p_i(x,y)$ and
$y\nmid p_i(x,y)$ are true, and there exist unique
$j\in \{1,\dots,k\}$ and $\lambda \in \F$ such that
$\hat{p_i}(x,y)= \lambda p_j(x,y)$. 

If $a=\sum a_it^i\in \F[t]$ is a nonzero polynomial of degree $d\ge 0$,
then the reciprocal polynomial $\sum a_it^{d-i}$ will be denoted
by $\hat a$, like in the case of two variables. A list
$a_1(t),\dots,a_k(t)\in \F[t]$ is \emph{reciprocally
closed} if for each $i\in \{1,\dots,k\}$ the polynomial $a_i(t)$
is not divisible by $t$, and  there exist unique
$j\in \{1,\dots,k\}$ and $\lambda \in \F$ such that
$\hat{a_i}(t)= \lambda a_j(t)$.

\begin{lem}\label{32}
Let $p_1(x,y),\dots,p_k(x,y)\in \F[x,y]$ and $a_1(t),\dots,a_r(t)
\in \F[t]$ be two reciprocally closed lists of polynomials.
Denote by $\Gamma$ the set of all nonzero roots of polynomials
$a_1,\dots,a_r$. Assume that 
\begin{equation}\label{e:imp}
p_i(0,c)=0\ \Longrightarrow \ c\in \Gamma \text{ or } c=0
\end{equation}
holds for all $i\in \{1,\dots,k\}$. 

Let $i,j\in \{1,\dots,k\}$ and $\lambda \in \F$ be
such that $p_j(x,y) = \lambda \hat {p_i}(x,y)$ and $i\ne j$. If
\begin{equation*}
\gcd(p_i(x,c),p_\ell(x,c)) = 1
\end{equation*}
holds for all nonzero $c\in \F\setminus\Gamma$ and all $\ell\ne i$, $1\le\ell\le k$,
then
\begin{equation*}
\gcd(p_j(x,c),p_h(x,c)) = 1
\end{equation*}
holds for all nonzero $c\in \F\setminus\Gamma$ and $h\ne j$, $1\le h \le k$.
\end{lem}

\begin{proof}
Suppose that $h\ne j$ and $c\in \F\setminus \Gamma$, $c\ne 0$, are such
that $p_j(x,c)$ and $p_h(x,c)$ have a common root in $\bar \F$, say
$\gamma$. Thus $p_j(\gamma,c) = p_h(\gamma,c) = 0$. By (3.12),
$\gamma \ne 0$. Since the list $p_1(x,y),\dots,p_k(x,y)$ is
reciprocally closed, there exists $\ell \ne i$ such that $p_\ell(x,y)$
is a scalar multiple of $\hat{p_h}(x,c)$. Since $p_j(x,y)$ is a multiple
of $\hat{p_i}(x,y)$ we have $p_i(\gamma\m,c\m)=0=p_\ell(\gamma\m,c\m)$
and hence $\gcd(p_i(x,c\m),p_\ell(x,c\m))\ne1$.
By the assumption on $p_i$ this cannot be true unless $c\m\in \Gamma$.
We shall refute the latter possibility by proving that if
$c\m\in \Gamma$, then  $c\in \Gamma$.
That follows straightforwardly from the assumption that the list
$a_1,\dots,a_r$ is reciprocally closed. Indeed, since $c\m \in \Gamma$,
there exists $s\in \{1,\dots,r\}$ such that $a_s(c\m) = 0$. There
also exists $m\in \{1,\dots,r\}$ such that $a_m$ is a scalar
multiple of $\hat{a_s}$. Because of that, $a_m(c) = a_m((c\m)\m) = 0$.
This implies that $c\in \Gamma$ since $\Gamma$ is defined as the set of
all nonzero roots of polynomials $a_1,\dots,a_r$.
\end{proof}

If $a(t) = t-\gamma$, $\gamma \ne 0$, then $\hat a(t) = 
-\gamma(t-\gamma\m)$. Hence the list of nonzero $c$ that
fulfil one of the conditions \eref{32}--\eref{311} may be
considered as a set $\Gamma$ of nonzero roots of a reciprocally
closed list of polynomials in one variable.

Now, remove $x$ and $x{-}1$ from the list of polynomials \eref{31} that 
are the input to \lref{32}.
The remaining polynomials can be interpreted as a list
$p_1(x,c), \dots, p_{13}(x,c)$ such that $p_1(x,y), \dots,p_{13}(x,y)$
is a reciprocally closed list of polynomials in two variables.
It is easy to verify that if $0$ or $1$ is a root of
any of the polynomials $p_i(x,c)$, $1\le i \le 13$, then 
$c$ fulfils \eref{32}. Polynomials $x$ and $x{-}1$ can be thus
excised from the subsequent discussion, and \lref{32}
may be used.

\lref{32} will also be applied to some sublists of 
$p_1(x,c), \dots, p_{13}(x,c)$ that are reciprocally closed. The
first such sublist are the linear polynomials occurring in \eref{31}
(with $x$ and $x{-}1$ being removed). These are
$x{-}c$, $x{-}1{-}c$, $x{+}1{-}c$, $(1{-}c)x-c$, $(1{+}c)x-c$,
$x-(2c{-}c^2)$ and $(1{-}2c)x+c^2$. The latter two polynomials
are equal to $g_2(x,c)$ and $g_4(x,c)$. The list of these linear
polynomials is square-free if there are no duplicates in the set of 
their roots
\begin{equation*} \label{e39o}
R(c) = \left\{c,c+1,c-1,\frac c{1-c},\frac c{1+c},c(2-c), 
\frac{c^2}{2c-1}\right\}.
\end{equation*}
The reciprocity yields the following pairs of roots:
\begin{equation}\label{e310o}
\left\{c+1,\frac c{1+c}\right \},\quad  \left\{c-1,\frac c{1-c}\right\}
\text{ and } \left\{c(2-c), \frac{c^2}{2c-1}\right \}.
\end{equation}

We now prove a sequence of lemmas which explore properties of the polynomials
\eref{31}.

\begin{lem}\label{33}
If $c\in \F$ satisfies \eref{32}--\eref{34}, then $|R(c)| = 7$.
\end{lem}
\begin{proof} 
If \eref{32} holds, then $c$ is not equal to any other element of
$R(c)$. Any equality
within the pairs in \eref{310o} would require that 
$c^2{+}c{+}1=0$ or $c^2{-}c{+}1=0$ or $2c(c{-}1)^2=0$. 
By \eref{32} and \eref{33}, none of these conditions hold.
Clearly, $c{+}1\ne c{-}1$. If $c{+}1 = c/(1{-}c)$, then $c^2{+}c{-}1=0$.
Furthermore, $c{+}1=c(2{-}c)$ $\Leftrightarrow$ $c^2{-}c{+}1=0$, and
$c{+}1=c^2/(2c{-}1)$ $\Leftrightarrow$ $c^2{+}c{-}1=0$. Hence $c{+}1$
is not equal to any other element of $R(c)$. By the reciprocity
relationship described in \lref{32}, $c/(1{+}c)$ is also not equal
to another element of $R(c)$. If $c{-}1$ is equal to $c(2{-}c)$, then
$c^2{-}c{-}1 =0$. If it is equal to $c^2/(2c{-}1)$, then $c^2{-}3c{+}1=0$.
\end{proof}

\begin{lem}\label{34}
Suppose that $c\in \F$ satisfies \eref{32} and \eref{35}. Then
none of the polynomials $f_i(x,c)$, $1\le i \le 4$, and $g_j(x,c)$,
$j \in \{1,3\}$, possesses a double root.
\end{lem}
\begin{proof} By a reciprocity argument similar to that of \lref{32} 
only $f_1(x,c)$, $f_2(x,c)$ and $g_1(x,c)$ need to be tested.
Discriminants of these polynomials are $(c+1)^2-4c^2 = (1-c)(3c+1)$,
$c(c-4(c-1)) = -c(3c-4)$ and $4(1-c)$. None of these may be
zero, by the assumptions on $c$.
\end{proof}

\begin{lem}\label{35}
If $c\in \F$ satisfies \eref{32}, \eref{33} and \eref{36},
then none of the elements of $R(c)$
is a root of $g_1(x,c)$ or $g_3(x,c)$.
\end{lem}
\begin{proof}
By \lref{32} it suffices to consider only the polynomial $h(x)=g_1(x,c)$.
Now, $h(c)  = c(c{-}1)$, $h(c\pm 1)=
c^2\pm 2c+1-2(c\pm 1) +c$ is equal to $c^2+c-1$ or $c^2-3c+3$,
while
\begin{equation*}
\frac{(1\pm c)^2}c h\left (\frac c{1\pm c}\right ) = c
-2(1\pm c) + (1\pm c)^2 = c^2+c-1
\end{equation*}
and $c\m h(c(2-c))=c(2-c)^2 - 2(2-c)+1=c^3-4c^2+6c-3=
(c-1)(c^2-3c+3)$. Finally,
\begin{equation*}
\frac{(2c-1)^2}c h\left (\frac {c^2}{2c-1}\right ) = c^3
-2c(2c-1) + (2c-1)^2 = c^3-2c+1=(c-1)(c^2+c-1).\qedhere
\end{equation*}
\end{proof}

\begin{lem}\label{norootsf}
If $c\in \F$ satisfies \eref{32}, \eref{35} and \eref{37}--\eref{311},
then none of the elements of $R(c)$
is a root of $f_i(x,c)$ for any $i=1,2,3,4$.
\end{lem}

\begin{proof}
The proof is very similar to that of \lref{35}, so we only give a summary.
By \lref{32}, it suffices to test the polynomials $f_1(x,c)$
and $f_2(x,c)$. Substituting an element
of $R(c)$ in place of $x$ always yields a polynomial from
the indicated list.
Note that $c^3-3c^2+4c-2 = (c-1)(c^2-2c+2)$, $3c^2-5c+2 =
(c-1)(3c-2)$ and $3c^3-7c^2 +5c -1 = (c-1)^2(3c-1)$.
\end{proof}

\begin{lem}\label{36}
Suppose that $c\in \F$ satisfies \eref{32} and \eref{35}. Then
for each $i\in \{1,3\}$ there exist at least three $j\in \{1,2,3,4\}$
such that $g_i(x,c)$ and $f_j(x,c)$ share no root in $\bar \F$.
\end{lem}

\begin{proof} Because of the reciprocity, $i=1$ may be assumed.
If $g_1(x,c)$ and $f_1(x,c)$ have a common root $x$, then $(c{+}1)x-c^2 =2x-c$,
and that yields $(c{-}1)x=c(c{-}1)$. If $g_1(x,c)$ and $f_2(x,c)$ have a common
root, then $cx -c^2+c = 2x-c$, which means that $(c{-}2)x = (c{-}2)c$.
If $g_1(x,c)$ and $f_3(x,c)$ have a common root, then $(c^2{+}c)x -c^2 = 2x-c$,
and $(c{-}1)(c{+}2)x = (c{-}1)c$. In such a case $c\ne -2$ and
$x = c/(c{+}2)$. The latter value is a root of $g_1(x,c)$ if and only
if $0=c^2-2c(c{+}2)+c(c{+}2)^2=c^2(c+3)$. Here, as earlier, the solutions
for $c$ are forbidden by the conjunction of \eref{32} and \eref{35}.
\end{proof}

\begin{lem}\label{37}
If $c\in \F$ satisfies \eref{32}, \eref{37} and \eref{38},
and if $1\le i < j \le 4$, 
then $f_i(x,c)$ and $f_j(x,c)$ share no common root in $\bar \F$.
\end{lem}

\begin{proof}
This is obvious if $(i,j) = (1,3)$. If $(i,j) = (2,4)$, then $c(x^2-1)=0$, so
$x\in \{-1,1\}$. Now, $f_2(1,c) = (c-1)^2 \ne 0$, while
$f_2(-1,c) = c^2+1 = -f_4(-1,c)$. This is why $c^2\ne -1$ has to be
assumed.

For the rest it suffices to test pairs $(1,2)$ and $(2,3)$,
by the reciprocity described in \lref{32}. If $cx+x-c^2 = cx - c^2+c$,
then $x=c$ and $f_2(c) = c(c-1) \ne 0$. If $cx -c^2 + c =
cx + c^2x -c^2$, then $x=c\m$ and $f_2(c\m) = c^{-2}-1+c^2-c
= c^{-2}(c^4-c^3-c^2+1) = c^{-2}(c-1)(c^3-c-1)$.
\end{proof}

We can now bring all the pieces together to prove the main result of this
section.

\begin{proof}[Proof of \tref{31}] Suppose 
that $c$ fulfils \eref{32}--\eref{311}. Besides
Lemmas~\ref{33}--\ref{37}
we also use that $g_1(x,c)$ and $g_3(x,c)$ share no root
in $\bar F$, which can be proved by a similar method to \lref{37}.


Let $p_1(x,c),\dots,p_k(x,c)$ be a nonempty sublist
of \eref{31} such that the product $p_1(x,c)\cdots p_k(x,c)$ is 
a square in $\bar\F[x]$. Let $J$ be the set of those $j\in\{1,2,3,4\}$
for which there exists $h\in \{1,\dots,k\}$ such that $f_j(x,c) = p_h(x,c)$.
The set $J$ must be nonempty, by Lemmas \ref{33}--\ref{35}. Since
$J$ is nonempty and Lemmas~\ref{34}, \ref{norootsf} and \ref{37} hold,
there must exist $i\in \{1,3\}$ such that $g_i(x,c) = p_h(x,c)$ for
some $h\in \{1,\dots,k\}$. Since $g_i(x,c)$ is not a scalar multiple
of $f_j(x,c)$ for $j\in J$, we must have $|J|\ge 2$, by Lemmas~\ref{34}
and~\ref{35}.
However, even that is not viable, given \lref{36}. 
\end{proof}

\section{When $-1$ is a nonsquare}\label{4}

Throughout this section $\F=\F_q$ will be a finite field of order $q\equiv
3\bmod 4$. We put $\Sigma = \Sigma(\F_q)$ and $S=S(\F_q)$. By \cref{13},
$|S| = |\Sigma| = (q^2{-}8q{+}15)/4$.  Define
$\bar S_{ij}^{rs}=S\setminus S_{ij}^{rs}$ and $T=\bigcap\bar S_{ij}^{rs}$,
where sets $S_{ij}^{rs}$ are characterised by \tref{211}, subject
to the assumption that \eref{25} holds. As we will see, \eref{25} holds
in all cases that are relevant for our calculations. The
aim of this section is to estimate the number $\sigma(q) =
\big|\{(a,b)\in \Sigma:Q_{a,b}$ is maximally nonassociative$\}\big|$.
By \pref{15}, $\sigma(q) = |T|=(q^2{-}8q{+}15)/4 - |\bigcup
S_{ij}^{rs}|$.
Put
\begin{align*} T_0& = \{(x,y)\in T: \chi(y-x) = 1\}; \\
T_{1,1} &= \{(x,y)\in T_0: \chi(1-y)=\chi(1-x)=1\}; \\
T_{1,-1} &= \{(x,y)\in T_0: \chi(1-y)=\chi(1-x)=-1\}; \\
T_2 &= \{(x,y)\in T_0: \chi(1-x)=-1 \text{ and } \chi(1-y) = 1\};
\end{align*}
and define $T_0'$, $T'_{1,1}$, $T'_{1,-1}$ and $T'_2$ by exchanging
$x$ and $y$. For example, $T_0'=\{(x,y)\in T:\chi(y-x) = -1\}$. 
Put also $T_1=T_{1,1}\cup
T_{1,-1}$ and $T_1'=T'_{1,1}\cup T'_{1,-1}$.

\begin{lem}\label{42}
$T=T_0\cup T_0'$, $T_0= T_1\cup T_2$, $T_0' = T_1' \cup T_2'$,
$T_1=T_{1,1}\cup T_{1,-1}$ and $T_1'=T'_{1,1}\cup T'_{1,-1}$.
All of these unions are unions of disjoints sets. 

Both of the mappings
$(x,y)\mapsto (y,x)$ and $(x,y)\mapsto (x\m,y\m)$ permute $T$.
Both of them exchange $T_1$ and $T_1'$, and $T_2$ and $T_2'$.
Furthermore, $(x,y)\mapsto (y,x)$ sends $T_{1,\eps}$
to $T'_{1,\eps}$, while $(x,y)\mapsto (x\m,y\m)$ sends
$T_{1,\eps}$ to $T'_{1,-\eps}$, for both $\eps\in \{-1,1\}$.
\end{lem}

\begin{proof}
Recall that by our definition of $S$, we have $1\notin\{x,y\}$ and $x\ne y$
for all $(x,y)\in T$.
By \pref{23} both $(x,y) \mapsto (y,x)$ and $(x,y)\mapsto (x\m,y\m)$
permute $T$. The effects of these
mappings are easy to verify. Note, for example, that
if $\eps = \chi(x{-}y)$, then $\chi(x\m-y\m) = \chi(y{-}x) = -\eps$.

To see that $T_0 = T_1\cup T_2$, note that there is no 
$(x,y) \in T_0$ with $\chi(1{-}x)=1$ and $\chi(1{-}y)=-1$. Indeed,
each such $(x,y)$ belongs to $S_{01}^{10}$.
\end{proof}

For $c\in\F_q$ define
$t_2(c)=\big|\{x\in\F_q:(x,c)\in T_2\}\big|$
and $t_{1,1}(c)=\big|\{x\in\F_q:(x,c)\in T_{1,1}\}\big|$.
In the next two propositions we seek estimates of these quantities.
In both results we will assume that $c$ fulfils condition
\eref{33}. Observe that under this assumption, $c^2{-}c{-}1\ne 0$ and for
all $x\in\F_q$ either $x\ne c{+}1$ or $x^2{-}x{-}1\ne 0$, and therefore
\eref{25} holds for $(x,y)=(x,c)$. This will enable us to use \tref{211}.

\begin{prop}\label{44}
Suppose that $c$ and $1-c$ are both nonzero squares in $\F_q$
and that $c$ fulfils conditions \eref{32}--\eref{311}.
Then $$|t_2(c)-25\cdot2^{-15}q|\le (\sqrt q+1)165/2+21.$$
\end{prop}

\begin{proof}
We estimate $t_2(c)$ by characterising the pairs $(x,c)$ in $T_2$.
For a fixed $c$ there are at most 21 values of $x$ that are roots
of any of the polynomials in \eref{31}. So at the cost of adding
a term equal to 21 to our eventual bound, we may assume
for the remainder of the proof that $x$ is not a root of any
polynomial in \eref{31}. Then
$\chi(x)=1=\chi(c)$ since $(x,c) \in S$ and 
$\chi(1{-}x)=\chi(c{-}1)=\chi(x{-}c)=-1$ by the definition of $T_2$.

From the definitions of $\bar S_{01}^{00}$, $\bar S_{10}^{00}$, $\bar S_{10}^{11}$
and $\bar S_{01}^{11}$ we deduce that
$\chi(g_1(x,c))=\chi(g_4(x,c))=-1$ and $\chi(g_2(x,c))=\chi(g_3(x,c))=1$.
Now, from $(x,c)\in\bar S_{01}^{01}$ we deduce that either
\begin{equation}\label{efirstalts}
\chi(x{-}1{-}c) = 1 \text{ or } \chi(x-xc-c)=1.
\end{equation}
In the former case, the requirement for $(x,c)$ to be in
$\bar S_{11}^{01}$ forces $\chi(f_1(x,c))=1$, whilst in the latter case
the requirement for $(x,c)$ to be in
$\bar S_{00}^{01}$ forces $\chi(f_4(x,c))=-1$. Of course, it is also
possible that both alternatives in \eref{firstalts} are realised.

Analogously, $(x,c)$ belongs to $\bar S_{10}^{10}$, so
\begin{equation}\label{esecondalts}
\chi(c{-}1{-}x) = 1 \text{ or } \chi(c-xc-x)=1.
\end{equation}
In the former case, the requirement for $(x,c)$ to be in
$\bar S_{11}^{10}$ forces $\chi(f_2(x,c))=-1$, whilst in the latter case
the requirement for $(x,c)$ to be in
$\bar S_{00}^{10}$ forces $\chi(f_3(x,c))=1$. 

Suppose that $i=1,\dots,9$ indexes the nine possibilities for the
quadruple
\begin{equation}\label{e:quad}
\big(\chi(c{-}1{-}x),\,\chi(c-xc-x),\,\chi(x{-}1{-}c),\,\chi(x-xc-c)\big)
\end{equation}
that are consistent with \eref{firstalts} and \eref{secondalts}.  In
each case, let $J_i$ denote the subset of $\{1,2,3,4\}$ consisting of
those indices $j$ for which $\chi(f_j)$ is forced.  Combining the
above observations, we see that there will be $4,4,1$ cases
respectively in which $|J_i|=2,3,4$.

By \tref{31} and our assumptions, the list
of polynomials in \eref{31} is square-free. We can hence apply
\tref{16} for each of the 9 possibilities for \eref{:quad}, prescribing
$\chi(p(x))$ for each polynomial $p(x)$ in \eref{31} except for
any $f_j$ with $j\notin J_i$.
We find that
\begin{align*}
|t_2(c)-4\cdot2^{-13}q-4\cdot2^{-14}q-1\cdot2^{-15}q|
&\le (\sqrt q+1)(4\cdot17+4\cdot19+1\cdot21)/2+21.
\end{align*}
The result follows.
\end{proof}

\begin{prop}\label{45}
Suppose that $c$ and $1-c$ are both nonzero squares in $\F_q$
and that $c$ fulfils conditions \eref{32}--\eref{311}.
Then $$|t_{1,1}(c)-25\cdot2^{-11}q|\le 96(\sqrt q+1)+21.$$
\end{prop}

\begin{proof} 
The proof is similar to that of \pref{44}.
Let us consider under which conditions a pair $(x,c)$ belongs to $T_{1,1}$,
where $x$ is not a root of any polynomial in \eref{31}.
For $(x,c)$ to belong to each of the sets 
$\bar S_{01}^{10}$,
$\bar S_{01}^{00}$, $\bar S_{10}^{00}$, $\bar S_{10}^{11}$,
$\bar S_{01}^{11}$, $\bar S_{11}^{01}$ and $\bar S_{00}^{10}$ it
is necessary and sufficient that
$\chi(x)=\chi(c)=\chi(1-c)=\chi(1-x)=\chi(c-x)=\chi(g_2(x,c)) = 1$ and
$\chi(g_4(x,c)) = -1$. Also for $(x,c)$ to be in $\bar S_{11}^{10}$ and
$\bar S_{10}^{10}$ requires that
\begin{align*}
&\chi(c{-}1{-}x) = -1\text{ or } \chi(f_2(x,c))=-1; \text{ and}\\
&\chi(c{-}1{-}x) = 1\text{ or } \chi(g_3(x,c))=-1 \text{ or } \chi(c-xc-x)=1.
\end{align*}
Both of these conditions are automatically 
satisfied if $\chi(c{-}1{-}x) = -1$ and 
$\chi(c-xc-x)=1$. Each of the other three possibilities for the pair
$\big(\chi(c{-}1{-}x),\chi(c-xc-x)\big)$ forces exactly one of the
conditions $\chi(f_2(x,c))=-1$ or $\chi(g_3(x,c))=-1$ to hold.

Similarly, for $(x,c)$ to be in $\bar S_{00}^{01}$ and
$\bar S_{01}^{01}$ requires that
\begin{align*}
&\chi(x{-}xc{-}c) = -1\text{ or } \chi(f_4(x,c))=-1; \text{ and}\\
&\chi(x{-}1{-}c) = 1\text{ or } \chi(g_1(x,c))=1 \text{ or } \chi(x-xc-c)=1.
\end{align*}
Both of these conditions are automatically 
satisfied if $\chi(x{-}1{-}c) = 1$ and 
$\chi(x-xc-c)=-1$. Each of the other three possibilities for the pair
$\big(\chi(x{-}1{-}c),\chi(x-xc-c)\big)$ forces exactly one of the
conditions $\chi(f_4(x,c))=-1$ or $\chi(g_1(x,c))=1$ to hold.

Suppose that $i=1,\dots,16$ indexes the sixteen possibilities for the
quadruple \eref{:quad}.  Let $K_i$ denote the subset of $\{f_2,f_4,g_1,g_3\}$
consisting of those polynomials $p$ for which $\chi(p)$ is forced.
Combining the above observations, we see that there will be $1,6,9$
cases respectively in which $|K_i|=0,1,2$. The values of $\chi(p)$ for
$p\in\{f_1,f_2,f_3,f_4,g_1,g_3\}\setminus K_i$ are unconstrained. Hence,
by applying \tref{16} for each of the 16 possibilities for \eref{:quad}
we find that
\begin{align*}
|t_{1,1}(c)-1\cdot2^{-9}q-6\cdot2^{-10}q-9\cdot2^{-11}q|
&\le (\sqrt q+1)(1\cdot9+6\cdot11+9\cdot13)/2+21.
\end{align*}
The result follows.
\end{proof}

We are now ready to prove the main result for this section.

\begin{thm}\label{46}
For $q\equiv3\bmod4$,
\[
\big|\sigma(q)-25(2^{-11}{+}2^{-16}) q^2\big|<138 q^{3/2} + 235q.
\]
\end{thm}

\begin{proof}
By \cite[Theorem~10.5]{Eva18} there are $(q-3)/4$ choices for $c\in\F_q$
such that both $c$ and $1-c$ are nonzero squares. At most
$1+4+1+3+2+3+0+2+3+3=22$ of these choices do not fulfil
conditions \eref{32}--\eref{311} of \tref{31}. (To see this, note that
$\chi(-1)=-1$ and that if $\chi(c)=\chi(1-c)=1$ then $\chi(1-1/c)=-1$,
which means that in any pair of reciprocal field elements, at most one
of the elements will be a viable choice for $c$. This is particularly
useful because of the many polynomials in \tref{31} which form
reciprocal pairs.)
Each $c$ that fails one of the conditions
\eref{32}--\eref{311} contributes between 0 and $(q-3)/2$ elements $(x,c)$
to $T$.
Putting these observations together with \pref{44} and \pref{45} we have
that
\begin{align*}
\big||T_2|-25\cdot2^{-15}q(q-3)/4\big|
&\le165(\sqrt{q}+1)(q-3)/8+21(q-3)/4+22(q-3)/2,\\
\big||T_{1,1}|-25\cdot2^{-11}q(q-3)/4\big|&\le96(\sqrt{q}+1)(q-3)/4+21(q-3)/4+22(q-3)/2.
\end{align*}
Next, notice that it follows from \lref{42} that
$|T_{1,1}|=|T'_{1,1}|=|T_{1,-1}|=|T'_{1,-1}|$ and $|T_2|=|T'_2|$  and
that $T$ is the disjoint union of $T_{1,1}$, $T'_{1,1}$, $T_{1,-1}$,
$T'_{1,-1}$, $T_2$ and $T'_2$. Hence
\begin{align*}
\big||T|-25(2^{-16}+2^{-11})q(q-3)\big|
&\le(q-3)\Big[(165/4+96)(\sqrt{q}+1)+195/2\Big]\\
&\le q\Big[138\sqrt{q}+939/4\Big].
\end{align*}
The result then follows from simple rearrangement.
\end{proof}

\begin{cor}\label{48}
Let $q$ run through all prime powers $\equiv 3 \bmod 4$.
Then $\lim \sigma(q)/q^2 = 25(2^{-11}{+}2^{-16})$.
\end{cor}

\section{When $-1$ is a square}\label{5}

Throughout this section $\F=\F_q$ will be a finite field of order
$q\equiv1\bmod4$. Our broad strategy for obtaining an estimate of
$\sigma(q)$ is similar to that used in \secref{4}.  For
$i,j,r,s\in \{0,1\}$ define $\bar S_{ij}^{rs}=S\setminus S_{ij}^{rs}$
and put $T = \bigcap \bar S_{ij}^{rs}$.  The set $T$ will again be
expressed as a disjoint union of sets the size of each of which can be
estimated by means of the Weil bound.
Let $\eps=\chi(x-y)$ and define
\begin{align*}
T_1 &= \{(x,y)\in T: \chi(1{-}x)=\chi(1{-}y)=-\eps\};\\
T_2 &= \{(x,y)\in T: \chi(1{-}x) = \eps\text{ and }
\chi(1{-}y) = -\eps\}; \text{ and}\\
T_2' &= \{(x,y)\in T: \chi(1{-}x) = -\eps\text{ and } \chi(1{-}y) = \eps\}.
\end{align*}
If $\rho_j\in \{-1,1\}$ for $1\le j \le 4$, then define
\begin{align*}
R(\rho_1,\rho_2,\rho_3,\rho_4)\ &=\{(x,y)\in T: \rho_j = 
\eps\,\chi(f_j(x,y))
\text{ for } 1\le j \le 4\};\\
R_1(\rho_1,\rho_2,\rho_3,\rho_4)&=T_1\cap R(\rho_1,\rho_2,\rho_3,\rho_4);
\text{ and } \\
R_2(\rho_1,\rho_2,\rho_3,\rho_4)&= T_2\cap R(\rho_1,\rho_2,\rho_3,\rho_4).
\end{align*}
We will write $R(\bar \rho)$ as a shorthand for $R(\rho_1,\rho_2,\rho_3,\rho_4)$
where $\bar \rho = (\rho_1,\rho_2,\rho_3,\rho_4)$. We record the following
basic facts about the sets just defined.

\begin{lem}\label{51}
Suppose $\rho_j\in\{-1,1\}$ for $1\le j\le4$.
The map $(x,y)\mapsto (y,x)$ induces bijections that show that
$|R_1(\rho_1,\rho_2,\rho_3,\rho_4)|=|R_1(\rho_2,\rho_1,\rho_4,\rho_3)|$
and $|T_2|=|T'_2|$. Hence $|T|=|T_1|+2|T_2|$.
The map $(x,y)\mapsto (x^{-1},y^{-1})$ induces bijections that show that
$|R_i(\rho_1,\rho_2,\rho_3,\rho_4)|=|R_i(\rho_3,\rho_4,\rho_1,\rho_2)|$
for $i\in\{1,2\}$. Also,
$R(\rho_1,\rho_2,-1,-1)=R(-1,-1,\rho_3,\rho_4)=\emptyset$.
\end{lem}

\begin{proof}
By \pref{23}, we know that $(x,y)\mapsto (x\m,y\m)$ permutes 
each of the sets $T_1$, $T_2$ and $T_2'$, while $(x,y)
\mapsto (y,x)$ permutes $T_1$ and swaps $T_2$ and $T_2'$.
This gives us a bijection between $T_2$ and $T_2'$. 
Note also that $T=T_1\cup T_2\cup T_2'$ since $\chi(1{-}x)=\chi(1{-}y)=\eps$
implies that $(x,y)\in S_{00}^{00}$. Hence $|T|=|T_1|+2|T_2|$.
The remaining claims about bijections follow directly from
the definitions of $f_1$, $f_2$, $f_3$ and $f_4$ in \eref{26}.

If $(x,y)\in R(\rho_1,\rho_2,-1,-1)$, then 
$\chi(f_j(x,y))=-\eps$ for both $j\in \{3,4\}$.
That implies $(x,y)\in S_{00}^{11}$, by \lref{24}. 
Hence, $R(\rho_1,\rho_2,-1,-1)=\emptyset$
and our bijection gives $R(-1,-1,\rho_3,\rho_4)=\emptyset$.
\end{proof}

Our aim is to use the $|R_i(\bar\rho)|$ to estimate the size of $T$.  We
should note that $T$ may be a proper superset of
$\bigcup_{\bar\rho}R(\bar\rho)$. The (small) difference arises from the
contribution to $T$ from roots of the polynomials $f_i$ (this
contribution will be accounted for later, when all roots are included
as an error term in our bounds).  \lref{51} 
reduces the number of $|R_i(\bar\rho)|$ that we need to estimate to
only those $\bar\rho$ shown in \Tref{T:rhosmu}. The final column of that table
shows the multiplicity $\mu$ that we need to use for each $|R_i(\bar\rho)|$
in order to obtain $|\bigcup_{\bar\rho}R_i(\bar\rho)|$. For example,
$R_1(1,1,1,-1)$ has $\mu=4$ because \lref{51} tells us 
that $$|R_1(1,1,1,-1)|=|R_1(1,1,-1,1)|=|R_1(1,-1,1,1)|=|R_1(-1,1,1,1)|.$$

\begin{table}[h]
$\begin{array}{r|rrrr|rrrr|c|c}
i & \rho_1 & \rho_2 & \rho_3 & \rho_4 & s_1 & s_2 & s_3 & s_4 & k_i(\bar\rho) & \mu\\
\hline
1 & 1 & 1 & 1 & 1 & 1 & 1 & 1 & 1 &4& 1\\
1 & 1 & 1 & 1 &-1 & 1 & 1 & 1 & 0 &3& 4\\
1 & 1 & -1& 1& -1 & 1 & 0 & 1 & 0 &2& 2\\
1 & 1 & -1&-1&  1 & 1 & 0 & 0 & 1 &2& 2\\
2 & 1 & 1 & 1 & 1 & 0 & 1 & 0 & 1 &2& 1 \\
2 & 1 & 1 & 1 &-1 & 0 & 1 & 0 & 0 &1& 2 \\
2 & 1 & -1& 1& -1 & 0 & 0 & 0 & 0 &0& 1 \\
2 & 1 & -1&-1&  1 & 0 & 0 & 0 & 1 &1& 2 \\
2 & 1 & 1 & -1 &1 & 0 & 1 & 0 & 1 &2& 2 \\
2 & -1 & 1& -1& 1 & 0 & 1 & 0 & 1 &2& 1 
\end{array}$
\medskip
\caption{\label{T:rhosmu}Values of $s(i,\bar\rho)$ and associated parameters.}
\end{table}

\begin{lem}\label{52}
Suppose that \eref{25} holds.
\begin{enumerate}
\item[(i)] If
$\chi(x{-}1{-}y)=\chi(x{-}xy{-}y)$, then $(x,y)\notin S_{01}^{01}$.
If $\chi\big((x{-}1{-}y)(x{-}xy{-}y)\big)=-1$, then there exist unique
$\lambda_1,\lambda_4\in \{-1,1\}$ such that $(x,y)\in S_{01}^{01}$
$\Leftrightarrow$ $\chi(g_j(x,y))=\lambda_j$ for $j\in \{1,4\}$.
\item[(ii)] If
$\chi(y{-}1{-}x)= \chi(y{-}xy{-}x)$, then $(x,y)\notin S_{10}^{10}$.
If $\chi\big((y{-}1{-}x)(y{-}xy{-}x)\big)=-1$, then there exist unique
$\lambda_2,\lambda_3\in \{-1,1\}$ such that $(x,y)\in S_{10}^{10}$
$\Leftrightarrow$ $\chi(g_j(x,y))=\lambda_j$ for $j\in \{2,3\}$.
\end{enumerate}
\end{lem}

\begin{proof}
Only case (i) needs to be proved, because of the $x\leftrightarrow y$
symmetry. If $\chi(x{-}1{-}y)= \chi(x{-}xy{-}y)$, then  
$(x,y) \notin S_{01}^{01}$ by \tref{210}.
If $\chi\big((x{-}1{-}y)(x{-}xy{-}y)\big)=-1$, then exactly one choice
of $(\chi(g_1(x,y)),\chi(g_4(x,y))$ makes $(x,y)$ an element
of $S_{01}^{01}$, again by \tref{210}.
\end{proof}

Consider $(x,y)\in R_i(\bar \rho)$ for a particular
$i\in\{1,2\}$ and $\bar\rho$. Membership of $R_i(\bar \rho)$ implies
values for $\chi(1{-}x)$, $\chi(1{-}y)$ and $\rho_j$ for
$1\le j \le 4$. Also, $(x,y)$ must belong to the sets
$\bar S_{11}^{01}$, $\bar S_{11}^{10}$, $\bar S_{00}^{10}$
and $\bar S_{00}^{01}$, which implies that some of the elements $x{-}1{-}y$,
$y{-}1{-}x$, $y{-}xy{-}x$ and $x{-}xy{-}y$ have to be nonsquares,
while for the others no such condition is imposed.  Record this into a
quadruple $s(i,\bar\rho)=(s_1,s_2,s_3,s_4)$, where $s_j \in \{0,1\}$
for $1\le j \le 4$. Here $s_1=1$, $s_2=1$, $s_3=1$ and $s_4=1$ mean
respectively that the presence of $(x,y)$ in $R_i(\bar \rho)$
forces $x{-}1{-}y$, $y{-}1{-}x$, $y{-}xy{-}x$ and $x{-}xy{-}y$ to be
nonsquare. For each $i$ and $\bar\rho$,
the value of the vector $s(i,\bar\rho)$ is given in
\Tref{T:rhosmu}. Furthermore, $k_i(\bar\rho)$ will be used to denote
the number of indices $j$ for which $s_j=1$ in $s(i,\bar\rho)$.

As an example consider $R_2(1,1,1,1)$. In this case $\chi(f_j(x,y))=\eps$
for all $j\in\{1,2,3,4\}$. Since $\chi(1{-}x) = \eps$, 
$(x,y)\notin S_{11}^{01}$ and $(x,y) \notin S_{00}^{10}$. Therefore
$s_1=s_3=0$. Since $\chi(1{-}y)=-\eps$ we must have $\chi(y{-}1{-}x)=-1$ 
if $(x,y)$ is to belong to $\bar S_{11}^{10}$. Therefore 
$s_2=1$. Similarly, $s_4=1$. 

For $c\in\F_q$ define
$t_1(c)=\big|\{x\in\F_q:(x,c)\in T_1\}\big|$
and $t_2(c)=\big|\{x\in\F_q:(x,c)\in T_2\}\big|$.
In the next two propositions we seek estimates of these quantities.
As in \secref{4}, we will assume that \eref{33} holds
which means that \eref{25} applies, enabling us to use \tref{210} and
\lref{52}.

\begin{prop}\label{p:t1}
Suppose that $c$ is a square satisfying conditions
\eref{32}--\eref{311}. Then
$$|t_1(c)-169\cdot2^{-14}q|\leq(\sqrt{q}+1)1161/2+21.$$
\end{prop}

\begin{proof}
Fix $c$ satisfying conditions \eref{32}--\eref{311} and consider a
candidate $(x,c)$ for membership in $T_1$.  As we did in \pref{44}, we
include the term 21 in our bound and then for the remainder of the
proof we may assume that $x$ is not a root of any polynomial
in \eref{31}.

Our goal is to estimate the $c$-slice of $R_1(\bar \rho)$ for each $\bar\rho$.
We start with a list of polynomials
that guarantee the presence of $(x,c)$ in $\bar S_{00}^{00}$,
$\bar S_{11}^{00}$, $\bar S_{00}^{11}$, $\bar S_{11}^{01}$,
$\bar S_{11}^{10}$, $\bar S_{00}^{10}$ and $\bar S_{00}^{01}$.
These polynomials are $x$, $1{-}x$, $c{-}x$, $f_j(x,c)$,
$1\le j \le 4$, and those of $x{-}1{-}c$, $c{-}1{-}x$, $c{-}cx{-}x$ 
and $x{-}cx{-}c$
for which the corresponding value of $s_j$ in 
$s(i,\bar \rho)$ is equal to $1$. In this way we will obtain a 
list of $7+k_1(\bar\rho)$
polynomials of cumulative degree $11+k_1(\bar \rho)$, for 
$k_1(\bar \rho)$ as shown in \Tref{T:rhosmu}. 

It only remains to ensure that $(x,c)$ is in $\bar S_{01}^{01}$ and
$\bar S_{10}^{10}$.  The $c$-slice of $R_1(\bar \rho)$ forks into
several disjoint subsets, according to \lref{52}. The forking induced
by $\bar S_{01}^{01}$ depends upon $(s_1,s_4)$, while the forking
induced by $\bar S_{10}^{10}$ depends upon $(s_2,s_3)$. It is thus
possible to describe only the former and obtain the latter by
exploiting the $(x,y)\leftrightarrow(x\m,y\m)$ symmetry between
$\bar S_{01}^{01}$ and $\bar S_{10}^{10}$.

If $s_1=s_4= 1$, then there is no forking since this suffices
to conclude that $(x,y)\notin S_{01}^{01}$.

If $s_1+s_4 = 1$, then one of $\chi(x{-}1{-}c)=-1$ and
$\chi(x{-}xc{-}c)=-1$ is mandated, and 
there are four forks. One of them specifies the
character of only one extra polynomial to ensure that
$\chi(x{-}1{-}c)=\chi(x{-}xc{-}c)=-1$.
Each of the other three forks imposes
restrictions on 
three polynomials, as it establishes first that
$\chi(x{-}1{-}c)=-\chi(x{-}xc{-}c)$ and then imposes values on
$\chi(g_1(x,c))$ and $\chi(g_4(x,c))$.  By \lref{52} there are three
possibilities to consider for the pair
$\big(\chi(g_1(x,c)),\chi(g_4(x,c))\big)$, which thus give us the three
forks.

The forking of the case $s_1+s_4 = 1$ will be recorded by $(1,1)\mid(3,4)^3$.
This means that the first fork needs one additional 
polynomial of degree one, while the other three forks
need three polynomials of cumulative degree $4$. 

If $s_1=s_4=0$, then there are seven forks. One of them
imposes that $\chi\big((x{-}1{-}c)(x{-}xc{-}c)\big)=1$
(which ensures that $\chi(x{-}1{-}c)=\chi(x{-}xc{-}c)$), while
each of the other six establishes first the (different)
values of $\chi(x{-}1{-}c)$ and $\chi(x{-}xc{-}c)$, and 
then the values of $\chi(g_1(x,c))$ and $\chi(g_4(x,c))$.
Symbolically, this gives $(1,2)\mid(4,5)^6$. 

Let us use $\bullet$ to express composition of two independent
forkings. Thus 
\begin{equation*}
(k_1,d_1)^{m_1}\mid\dots \mid(k_a,d_a)^{m_a}
\bullet (k'_1,d'_1)^{m'_1}\mid \dots \mid (k'_b,d'_b)^{m'_b}
\end{equation*}
is a list of alternatives $(k_i+k_j', d_i+d_j')^{m_im'_j}$, where
$1\le i \le a$ and $1\le j \le b$.

Our observations above allow us to symbolically describe 
polynomial lists for each of the sets $R_1(\bar \rho)$. We have 
\begin{align*}
R_1(1,1,1,1)\colon &(11,15), \\
R_1(1,1,1,-1)\colon &(10,14)\bullet (1,1)\mid(3,4)^3 = (11,15)\mid(13,18)^3, \\
R_1(1,-1,1,-1)\colon & (9,13)\bullet (1,1)\mid(3,4)^3 \bullet (1,1)\mid (3,4)^3 
= (11,15)\mid (13,18)^6 \mid (15,21)^9, \\
R_1(1,-1,-1,1)\colon & (9,13)\bullet (1,2)\mid(4,5)^6 = (10,15)\mid (13,18)^6. 
\end{align*} 
Combining this information with the last column of \Tref{T:rhosmu}, we
reach a symbolic description of the polynomials contributing to
$t_1(c)$ that contains 
$(10,15)$ with multiplicity $2$,
$(11,15)$ with multiplicity $1+4+2=7$,
$(13,18)$ with multiplicity $4\cdot 3 + 2\cdot 6+2\cdot 6 = 36$, and
$(15,21)$ with multiplicity $2\cdot 9=18$. In each case, the list of
polynomials involved is square-free, by \tref{31}. Hence, we may apply
\tref{t:Weil} to find that
\begin{equation*}
|t_1(c)-\alpha_1q|\le(\sqrt{q}+1)D_1/2+21,
\end{equation*}
where $\alpha_1 = 2\cdot 2^{-10}+7\cdot 2^{-11} +
36\cdot 2^{-13} + 18\cdot 2^{-15}=169\cdot2^{-14}$, and 
the cumulative degree of our polynomials is
$D_1=9\cdot15+36\cdot18+18\cdot21=1161$.
\end{proof}

\begin{prop}\label{p:t2}
Suppose that $c$ is a square satisfying conditions
\eref{32}--\eref{311}. Then
$$|t_2(c)-49\cdot2^{-11}q|\leq(\sqrt{q}+1)4455/2+21.$$
\end{prop}

\begin{proof}
The proof follows the same lines as that of \pref{p:t1}. The symbolic
description of the forks is
\begin{align*}
R_2(1,1,1,1)\colon &(9,13)\bullet (1,1)\mid(3,4)^3 \bullet (1,1)\mid(3,4)^3 
 = (11,15)\mid (13,18)^6 \mid (15,21)^9 \\
R_2(1,1,1,-1)\colon & (8,12)\bullet (1,1)\mid(3,4)^3 \bullet (1,2)\mid(4,5)^6
= (10,15)\mid (12,18)^3\mid (13,18)^6 \mid (15,21)^{18} \\
R_2(1,-1,1,-1)\colon & (7,11)\bullet (1,2)\mid(4,5)^6\bullet (1,2)\mid(4,5)^6
= (9,15)\mid (12,18)^{12} \mid (15,21)^{36} \\
R_2(1,-1,-1,1)\colon & (8,12)\bullet (1,1)\mid(3,4)^3 \bullet (1,2)\mid(4,5)^6
= (10,15)\mid (12,18)^3\mid (13,18)^6 \mid (15,21)^{18} \\
R_2(1,1,-1,1)\colon & (9,13) \bullet (1,1)\mid(3,4)^3 \bullet(1,1)\mid(3,4)^3
= (11,15)\mid (13,18)^6 \mid (15,21)^9 \\
R_2(-1,1,-1,1)\colon & (9,13) \bullet (1,1)\mid(3,4)^3 \bullet(1,1)\mid(3,4)^3
= (11,15)\mid (13,18)^6 \mid (15,21)^9 
\end{align*} 
Combining this information with the last column of \Tref{T:rhosmu}, we
reach a symbolic description of the polynomials contributing to
$t_2(c)$ that contains 
$(9,15)$ with multiplicity $1$,
$(10,15)$ with multiplicity $2\cdot1+2\cdot1=4$,
$(11,15)$ with multiplicity $1+2\cdot1+1=4$,
$(12,18)$ with multiplicity $2\cdot 3 + 12 + 2\cdot 3 = 24$, and
$(13,18)$ with multiplicity $6+2\cdot 6+2\cdot 6+2\cdot 6+6 = 48$, and
$(15,21)$ with multiplicity $9+2\cdot18+36+2\cdot18+2\cdot9+9=144$.
Combining \tref{31} and \tref{t:Weil}, we find that
\begin{equation*}
|t_2(c)-\alpha_2q|\le(\sqrt{q}+1)D_2/2+21,
\end{equation*}
where $\alpha_2 = 2\cdot 2^{-9}+4\cdot 2^{-10}+4\cdot 2^{-11}
+24\cdot 2^{-12}+48\cdot 2^{-13} + 144\cdot 2^{-15}=49\cdot2^{-11}$, and 
the cumulative degree of our polynomials is
$D_2=9\cdot15+72\cdot18+144\cdot21=4455$.
\end{proof}

We are now ready to prove the main result for this section.

\begin{thm}\label{53}
If $q\equiv1\bmod4$, then
\[
\big|\sigma(q)-953\cdot 2^{-15}\,q^2\big|<2518q^{3/2}+2623q.
\]         
\end{thm}

\begin{proof}
There are $(q-3)/2$ choices for a square $c\in\F_q$ satisfying
$c\notin\{0,1\}$. At most
$49$ of these choices do not fulfil conditions \eref{32}--\eref{311}
of \tref{31}. Each $c$ that fails one of the conditions
\eref{32}--\eref{311} contributes between 0 and $(q-3)/2$ elements $(x,c)$
to $T$. Putting these observations together with \pref{p:t1} and
\pref{p:t2} we have that
\begin{align*}
\big||T_1|-169\cdot2^{-14}q(q-3)/2\big|
&\le1161(\sqrt{q}+1)(q-3)/4+21(q-3)/2+49(q-3)/2,\\
\big||T_2|-49\cdot2^{-11}q(q-3)/2\big|
&\le4455(\sqrt{q}+1)(q-3)/4+21(q-3)/2+49(q-3)/2.
\end{align*}
Next, by \lref{51} we know that $|T|=|T_1|+2|T_2|$, so
\begin{align*}
\big||T|-(169\cdot2^{-15}+49\cdot2^{-11})q(q-3)\big|
&\le(q-3)\Big[(1161/4+ 4455/2)(\sqrt{q}+1)+105\Big]\\
&<q\Big[2518\sqrt{q}+10491/4\Big].
\end{align*}
The result then follows from simple rearrangement.
\end{proof}

\begin{cor}\label{54}
Let $q$ run through all prime powers that are $1 \bmod 4$.
Then $\lim \sigma(q)/q^2 = 953/2^{15}$.
\end{cor}

\section{Conclusions}\label{6}
Theorems~\ref{46} and~\ref{53} give formulas that can be used
as estimates of $\sigma(q)$ for large $q$. We did not work hard to
optimise the constants in the bounds. Even if we had,
the number of applications of the Weil bound is too big to allow
the estimates to be useful for small $q$.

The proof of existence of maximally nonassociative quasigroups
for small orders was obtained in \cite{dw} by
considering $(a,b)\in \Sigma$ that satisfy some additional condition
$t(a,b)=0$. It might be of interest to investigate, along these lines,
all cases $t(x,y) = 0$ when $t$ is one of the polynomials $x{-}1{-}y$,
$x{-}xy{-}y$, $y{-}1{-}x$, $y{-}xy{-}x$, $f_j(x,y)$ and $g_j(x,y)$ for
$1\le j \le 4$.

\subsection*{Acknowledgement}
This work was supported in part by Australian Research Council grant
DP150100506.

\end{document}